\documentclass[12pt]{article}
\usepackage{amssymb,amsmath,latexsym, amsthm,enumerate,  amsfonts}
\usepackage{wrapfig, tikz}
\usetikzlibrary{matrix, calc}
\usepackage{blindtext}
\title{} \author{} \date{}
\usepackage[pdftex, pdfstartview=FitH]{hyperref} % Use LaTeX and then DVI2PDF,
\usepackage[paperwidth=210mm, paperheight=297mm,  hmargin={15mm,15mm}, vmargin={15mm,15mm} ]{geometry}

\numberwithin{equation}{section} %You may omit this line if you want numbering as (1)

    \usepackage{wrapfig, tikz}
    \usepackage{tikz-cd}

\newcommand{\Cl}{\mathop{\mathrm {Cl}}\nolimits}
\newcommand{\CL}{\mathop{\mathrm {CL}\Gamma}\nolimits}
\newcommand{\Int}{\mathop{\mathrm {Int}}\nolimits}

 \newtheorem{thm}{Theorem}
 \newtheorem{cor}{Corollary}
 
 \newtheorem{prop}{Proposition}
 \theoremstyle{definition}
 \newtheorem{defn}{Definition}
  \newtheorem{ex}{Example}
 \theoremstyle{remark}
 
  \newtheorem*{que}{Question}

\begin{document}
\thispagestyle{empty}
% PRAVLJENJE NASLOVA
\begin{center}
{\large \bf  Various types of continuity and their interpretations in ideal topological spaces\footnote{This research was supported by the Science Fund of the Republic of Serbia, Grant No. 7750027: Set-theoretic, model-theoretic and Ramsey-theoretic phenomena in mathematical structures: similarity and diversity–SMART
}

} \vspace*{3mm}

{\bf Anika Njamcul\footnote{Department of Mathematics and Informatics, Faculty of Sciences, University of Novi Sad, Serbia, e-mail: \href{mailto:anika.njamcul@dmi.uns.ac.rs}{anika.njamcul@dmi.uns.ac.rs}}}
and
{\bf Aleksandar Pavlovi\'c\footnote{Department of Mathematics and Informatics, Faculty of Sciences, University of @Novi Sad, Serbia, e-mail: \href{mailto:apavlovic@dmi.uns.ac.rs}{apavlovic@dmi.uns.ac.rs}}}
\end{center}

\begin{abstract}

This paper is a continuation of work started in  \cite{njampavcont} on preserving continuity in ideal topological spaces. We will deal
with $\theta$-continuity and weak continuity and give their translations in ideal topological spaces. As  consequences of those results, we will prove that every $\theta$-continuous function is continuous if topologies are generated by $\theta$-open sets and  we will give an example of weakly continuous function which is not  $\tau_\theta$-continuous. This will complete the diagram of relations between continuous, $\tau_\theta$-continuous, $\theta$-continuous, weakly continuous and faintly continuous functions.
\\[2mm] {\it AMS Mathematics  Subject Classification $(2010)$}:
54A10, %Several topologies on one set (change of topology, comparison of topologies, lattices of topologies)
54A05,  %Topological spaces and generalizations (closure spaces, etc.)
54C05, %Continuous maps
54C08, %Weak and generalized continuity
54B99, %None of the above, but in this section
54E99 %None of the above, but in this section
\\[1mm] {\it Key words and phrases:} ideal topological space, local function, local closure function, $\theta$-open sets,
 $\theta$-closure,  continuity, $\theta$-continuous function, weakly continuous function, faintly continuous function

\end{abstract}

\section{Introduction}

Continuity is almost as old as general topology. Both notions are firstly mentioned by Fr\`echet, topological structure in 1906 \cite{Fre}, and  continuity in 1910 \cite{Frechet}. The importance of continuity in general topology is not needed to be explained here. Later, several modifications of continuity were defined. Some of them are $\theta$-continuity, weakly continuity, faintly continuity, almost-continuity, and many others.

 It is interesting that  $\theta$-continuity was defined before $\theta$-open sets. It was done by Fomin \cite{Fomin} in 1942. Later, after  Veli\v{c}ko  \cite{Veli} introduced $\theta$-open and $\theta$-closed sets, it turned out that those notions have some connections  with $\theta$-continuity.
Topology defined by $\theta$-open sets, the $\theta$-topology, was later studied by Hermann   \cite{Herrmann} and by Foroutan, Ganster and Steiner \cite{FGS}.
Weakly continuous function were  first mentioned by Levine in 1961 in \cite{Lev}. There he proved that weakly continuous function which is also weakly* continuous is continuous and vice versa.

During the history, some unintended overlapping occurred.
For example, closure continuity was introduced by Andrew and Whittlesy \cite{AnWh} in 1966. and it turned out that it is equivalent to $\theta$-continuity. Almost continuous mapping were presented the same year by Husain \cite{Husain}  and, by the same name, but with a  slightly different definition,  by Singal and Singal  \cite{SingSing}. Different forms of faintly continuous functions can be found in \cite{LongHerr} and \cite{NoiPopa}. Also, some weak forms of continuity were mentioned by Espelie and Joseph in \cite{SolvJos}.

%The study of $\theta$-closed sets was continued by Dickman and Porter \cite{DP} in 1975 proving that a $H$-closed set can not be a countable union of nowhere dense $\theta$-closed sets. Also they proved that the Martin's axiom is equivalent to the statement that every $H$-closed space with ccc is not a union of less than continuum many $\theta$-closed nowhere dense sets.

% In 2008, Kuyucu,  Noiri,  and \"{O}zkurt \cite{KNO}  investigated a connection between $\theta$-continuous mappings and  a class of mappings into ideal topological spaces, denoted by $w-I-$continunous mappings, introduced 2004 by A\c{c}ikg\"{o}z,  Noiri  and Y\"{u}ksel in \cite{ANY}. In 2018, Islam and Modak \cite{IM} further generalized the notion of local closure function, defining local semi-closure function.

 Kuratowski  was the first who considered ideals in general topology. In 1933   \cite{KURold, KUR} he defined  the local function, generalization of closure by an ideal.  About a decade later,    Vaidyanathaswamy continued the research on this topic in  \cite{Vaid}. Through the years ideal topological spaces became an interesting topic in topology, measure theory, etc. (see Freud \cite{Freud}, Scheinberg \cite{Sein}).
 %generalized Cantor-Bendixson theorem using ideal topological space.
%  applied ideals in the measure theory.
One of the most thorough papers on the local function and ideal topological spaces in general was written by Jankovi\'c and Hamlett \cite{JH} in 1990. This survey paper was used later as a basis for further research, mostly for studying modifications of the local function. Thus in 2013, Al-Omari and Noiri \cite{ON} introduced the local closure function as a generalization of $\theta$-closure in ideal topological spaces. In the same paper they mentioned two new topologies obtained from the starting topology using the local closure function.

 New variations of continuity were also defined in ideal topological spaces. Such examples can be found in the most recent works of Al-Omeri and Noiri  \cite{WadeiNoiri}, and of Powar,  Mishra and Bhadauria \cite{Powar_2021}. However, our work will consider some basic aspects of types of continuities and their natural interpretation in ideal topological spaces.

  %A\c{c}ikg\"{o}z,  Noiri  and Y\"{u}ksel wrote about $\omega$-I-continuous functions in \cite{ANY}, while in \cite{KNO} Noiri with  Kuyucu and \"{O}zkurt investigated more properties of these functions. % Many other types of continuity were defined since.

In Section 2 we will give basic definitions and notations. Also, we will give definitions of several topologies obtained in ideal topological spaces in which we will work. In Section 3 we will give definitions of continuity and its various types and present current state of results considering relations between those types of continuity presented as a diagram.  In the following two sections we will give results obtained as the continuation of the research started in  \cite{njampavcont} on preserving continuity in ideal topological spaces.  Section 4 is reserved for results concerning $\theta$-continuity and its consequences in ideal topological spaces. We will give a sufficient condition   for ideals in order to $\theta$-continuous function in topologies without ideals becomes continuous in $\sigma$, topology obtained by the local closure function.   At the end of this  section we will prove that $\theta$-continuity implies continuity in topologies consisting of $\theta$-open sets, $\tau_\theta$-topology, which will add a new arrow on the diagram. In Section 5 we will deal with weakly continuous functions and consequences in ideal topological spaces. A condition on ideals when weakly continuous functions becomes, in ideal topological spaces, a continuous between $\tau^*$ and $\sigma$ topologies, will be given.  As a direct consequence of those results is an already known result that weak continuity implies faint continuity. We will prove that,  in case when at least one of sets ( set originals or of images) is finite, weak continuity implies continuity in topology of $\theta$-open sets. Finally, we will give an infinite example of weakly continuous function which is not continuous in topology of $\theta$-open sets, proving that those two types are incomparable in general. This example will complete the diagram in the sense that no new arrows can be added.

\section{Basic definitions}

By  $\langle X, \tau\rangle$ we will  denote a  topological space,  $\tau(x)$ will be the family of open neighbourhoods at the point $x$.   Closure of the set $A$ will be written as  $\Cl_{\tau}(A)$  or, if it is clear, just by $\Cl({A})$. Similarly,  the  interior of $A$  will be denoted by $\Int(A)$ or $\Int_{\tau}(A)$.
%If a set $A$ is both open and closed, it is called \textbf{clopen}.
An important part of this paper will be dedicated to \textbf{$\theta$-topology}. This topology $\tau_{\theta}$ consists of all $\theta$-open sets: we say that a set $U$ is \textbf{$\theta$-open} if $$\forall x \in U ~  \exists V \in \tau(x)~ \Cl(V)\subseteq U.$$
$\Int_{\tau_\theta}(A)$ will denote the \textbf{interior in the topology of $\theta$-open sets}.  It is obvious that  $\tau_\theta \subseteq \tau$.
%=\bigcup \{U: U \subset A, U\mbox{ is }\theta\mbox{-open}\}$.
 $U$ is $\theta$-open if and only if $\Int_\theta(U)=U$.  Naturally, a set $A$ is $\theta$-closed if its complement $X \setminus A$ is $\theta$-open.

 \textbf{$\theta$-closure} $\Cl_\theta(A)$  is an operator in the starting topology. It is  defined by
$$\Cl_\theta(A)=\{x \in X: \Cl(U) \cap A \neq \emptyset\mbox{ for each }U \in \tau(x)\}.$$  A set $A$ is \textbf{$\theta$-closed} if and only if  it is equal to its $\theta$-closure.
It is important to notice that  $\theta$-closure of a set does not have to be  $\theta$-closed, but it is always a closed set. We have $\Cl (A) \subseteq \Cl_\theta(A)$, for each set $A$. In order to distinguish closure in $\tau_\theta$ from the operator $\Cl_\theta$, the prior will be denoted by $\Cl_{\tau_\theta}$.

% If $\langle X, \tau\rangle$ is regular space, then every open set is $\theta$-open and therefore $\tau=\tau_\theta$. More about $\theta$-open amd $\theta$-closed sets cab be found in \cite{theta} ????????

 We will use small Greek letters $\alpha, \beta, \gamma, \ldots, \omega, \ldots$ to denote ordinals. The family of all ordinals is denoted by $ON$. Letters  $\lambda$ and $\kappa$ will be used for cardinals, while $\aleph_0$  is  the first infinite cardinal.

 %For the first infinite cardinal we will write $\aleph_0$ and the set of naturals (including 0) will be $\omega$.   For a set $A$ and a cardinal $\kappa$,   $[A]^{\kappa}$, $[A]^{<\kappa}$ and $[A]^{\leq \kappa}$  are  the families of all subsets of $A$ of cardinality $\kappa$, less than $\kappa$ and less than or equal to $\kappa$, respectively.

An \textbf{ideal} on  a nonempty set $X$ is a family $\mathcal{I}\subset P(X)$ such that

{\leftskip 5mm (1) $\emptyset \in \mathcal{I}$,

(2) If $A \in  \mathcal{I}$  and $B \subseteq A$, then $B \in \mathcal{I}$,

(3) If $A, B \in \mathcal{I}$, then $A \cup B \in \mathcal{I}$.

}

 %Let us notice that, by this definition,  $P(X)$  is an ideal.
If    $\langle X, \tau\rangle$ is a  topological space, then an \textbf{ideal topological space} is a triple $\langle X, \tau, \mathcal{I}\rangle$ .

In an ideal topological space $\langle X, \tau, \mathcal{I}\rangle$ ,  the \textbf{local function} (see \cite{KUR}) can be defined  as follows
$$A^*_{(\tau, \mathcal{I})}=\{x \in X: A \cap U \not \in \mathcal{I} \mbox { for each }  U \in \tau(x)\}.$$
If it is clear which topology and ideal are considered, we write briefly $A^*$. It is monotone operator and $(A^*)^*\subseteq A^*$. Clearly, if $\mathcal{I}=\{\emptyset\}$, then $A^*=\Cl (A)$.

Basic properties of the local function can be found in the survey paper of Jankovi\'c and Hamlett   \cite{JH}.

%{\leftskip 5mm (1) $A \subseteq B \Rightarrow A^* \subseteq B^*$;

%(2) $A^*=\Cl(A^*)\subseteq \Cl(A)$;

%(3) $(A^*)^*\subseteq A^*$;

%(4) $(A\cup B)^*=A^* \cup B^*$

%(5) If $I \in \mathcal{I}$, then $(A\cup I)^*=A^*=(A\setminus I)^*$.

%}

Using the local function, a new topology $\tau^{*}(\mathcal{I})$ can be defined using the closure operator $\Cl^*(A)=A \cup A^*$.
Therefore,  $\tau^*(\mathcal{I})$ can be described as
$$\tau^*(\mathcal{I})=\{U \subseteq X: \Cl^*(X \setminus U)=X \setminus U\}.$$
Note that $\tau \subseteq \tau^*\subseteq P(X)$.

Several modifications of the local function were studied though the history. We will deal with the one given by  Al-Omari and   Noiri \cite{ON}. They defined the
 \textbf{local closure function} as a generalization of $\theta$-closure in ideal topological spaces.  The local closure function   in an ideal topological space $\langle X, \tau, \mathcal{I}\rangle$ is defined as
$$\Gamma_{(\tau,\mathcal{I})}(A)=\{x \in X: \Cl(U) \cap A \not \in \mathcal {I}\mbox{ for each }U \in \tau(x)\}.$$

 If the topology and the ideal are given, we write briefly $\Gamma(A)$. It is monotone operator, but there is no general relation between $A$ and $\Gamma(A)$, and it is not idempotent. Notice that if $ \mathcal {I}=\{\emptyset\}$ then,  for each set $A$, we have $\Gamma(A)=\Cl_\theta(A)$.

Some basic properties of the local closure function can be found in  \cite{ON}, and further analysis of its properties and relations with the local function in \cite{pav} and \cite{njampav}.

%{\leftskip 5mm (1) $A^*\subseteq  \Gamma(A)$;

%(2) $\Gamma(A)=\Cl(\Gamma(A))\subseteq \Cl_\theta(A)$;

%(3) $\Gamma(A \cup B)=\Gamma(A) \cup \Gamma(B)$;

%(4) $\Gamma(A \cup I)=\Gamma(A) =\Gamma(A \setminus I)$ for each $I \in \mathcal{I}$.

%\noindent Let us notice that there are no general relation  between $\Gamma(\Gamma(A))$ and $\Gamma(A)$, so  both inclusions and equality are possible.

Al-Omari and Noiri  \cite{ON}  also studied  a  variant of $\theta$-interior in ideal topological spaces. They denoted this operator  by $\psi_\Gamma(A)$ and defined it by $$\psi_\Gamma(A)=X \setminus \Gamma(X \setminus A).$$
 Using  $\psi_\Gamma(A)$  they defined a new topology  $\sigma$ using the operator $ \psi_\Gamma$:
$$A \in \sigma \Leftrightarrow A\subseteq \psi_\Gamma(A).$$
$F$ is a closed set in the topology $\sigma$ iff  $\Gamma(F) \subseteq F$.
It is important to point out that $\tau_\theta \subseteq \sigma$, and if  $ \mathcal {I}=\{\emptyset\}$, we have $\tau_\theta = \sigma$.

Since we are dealing with functions,we will always deal with two topologies. To distinguish them, sometimes we will put index of the set next to the topology, like $\tau_X$, or $\sigma_Y$. But, when it is clear about what is the carrier set of the topology we are talking about, that index will be omitted, especially when the name of the topology has to be part of the closure or interior operator.

%%%%%%%%%%%%%%%%%%%%%%%%%%%%%%%%%%%%%%%%%%%%%%%%%%%%%%%% novo

\section{Several types of continuity}

In this section we will give definitions of various types of continuity and their known relations. All are defined in classical topological spaces without ideals.

The notation is standard. If $f:X \to Y$, for $A \subseteq X$ and $B \subseteq Y$, direct image of the set $A$ is defined by $f[A]=\{f(x):x \in A\}$ and preimage of $B$ is defined  by $f^{-1}[B]=\{x \in X :f(x) \in B \}$.

The following definition belongs to the folklore of general topology.
\begin{defn}
A  function  $f:X \to Y$  is \textbf{continuous } at  the  point  $x\in  X$   if  and  only  if  for  each  neighbourhood  $V$  of  $f(x)$  there  is a  neighbourhood  $U$   of  $x$  such  that  $$f[U]\subseteq V.$$   $f:X\to Y$  is  continuous  if  and  only  if  $f$  is    continuous  at  each  point  $x  \in   X$.
\end{defn}

\begin{prop}\label{continuityeq}
\cite[Proposition 1.4.1]{eng} For  $f:\langle X, \tau_X\rangle \to \langle Y, \tau_Y\rangle$ the following conditions are equivalent

a) $f$ is continuous.

b) For each $O\in  \tau_Y$ we have $f^{-1}[O]\in \tau_X$.

c) For each $A\subseteq X$ we have $f[\Cl(A)]\subseteq \Cl(f[A])$.

d) For each $B\subseteq Y$ we have $\Cl(f^{-1}[B])\subseteq {f^{-1}[\Cl(B)]}$.

e) For each $B\subseteq Y$ we have ${f^{-1}[\Int(B)]}\subseteq \Int({f^{-1}[{B}]})$.

\end{prop}

\begin{defn}[Levine, \cite{Lev}]
	A  function  $f:X \to Y$  is \textbf{weakly  continuous}  at  the  point  $x\in  X$   if  and  only  if  for  each  neighbourhood  $V$  of  $f(x)$  there  is a  neighbourhood  $U$   of  $x$  such  that  $f[U]\subseteq \Cl(V)$.  A  function  $f:X\to Y$  is  weakly  continuous  if  and  only  if  $f$  is  weakly  continuous  at  each  point  $x  \in   X$.
\end{defn}

An equivalent condition for weak continuity can be given in terms of preimage. 

\begin{thm} [Levine, \cite{Lev}]  A  function  $f:X\to  Y$  is  \textbf{weakly   continuous}   if  and   only  if $f^{-1}[V] \subseteq \Int(f^{-1}[\Cl[V]]) $  for  each  open  subset  $V$  of   $Y$.
\end{thm}

%\begin{defn}\cite{SingSing}  A  function  $f:X \to     Y$  is almost  continuous  (in the sense of Singal and Singal) if  and  only  if  for  each  $x  \in  X$  and  for  each  neighbourhood  $V$  of  $f(x)$ exists neighbourhood $U$ of $x$ such that  $f[U] \subseteq \Int \Cl(V)$.
%\end{defn}

%\begin{thm}\cite{SingSing} If a function is weakly continuous, then it is almost continuous in the sense of Singal and Singal.
%\end{thm}

%\begin{defn}\cite{Husain}  A  function  $f:X \to     Y$  is almost  continuous  (in the sense of Husain) if  and  only  if  for  each  $x  \in  X$  and  for  each  neighbourhood  $V$  of  $f(x)$  $\Cl(f^{-1}[ V])$  is  a  neighbourhood  of  $x$.
%\end{defn}

%\begin{thm} \cite{Husain}  A  function  $f  : X  \to    Y$  is  almost  continuous  if and  only  $f^{-1}[V] \subseteq  \Int(Cl(f^{-1}[ V]))$  for  each  open  subset  $V$  of   $Y$.
%\end{thm}
%\begin{defn} \cite{Lev} A  function  $f:X \to Y$   is  weak*  continuous  if  and  only  if  $f^{-1}[\partial(V)]$   is  closed   in   $X$    for   each   open   subset   $V$   of   $Y$  where   $\partial(V)  =  \Cl(V)-\Int(V)$  is  the  frontier  (or  boundary)  of   $V$.
%\end{defn}
%Levine's  decomposition   theorem  \cite{Lev}  states  that  a  function  $f:X \to  Y$  is  continuous   if  and  only  if  $f   $ is  weakly  continuous   and   weak*   continuous.

\begin{defn} [Fomin, \cite{Fomin}] A  function   $f:  X\to  Y$ is \textbf{$\theta$-continuous} in $x_0 \in X$ iff for each open neighbourhood $V$ of $f(x_0)$ there exists open neighbourhood $U$ of $x_0$        such that $f[\Cl(U)] \subseteq \Cl(V)$. The same definition is given   in \cite{AnWh}, but there it is called \textbf{closure continuity}.
\end{defn}

It is important to mention that $\theta$-continuity is not the same as continuity in topologies of $\theta$-open sets. Therefore, to make a difference, the second type of continuity we will call \textbf{$\tau_\theta$-continuity}. The following result gives  a sufficient condition for preserving $\theta$-continuity when topology $\tau$ on the domain is replaced with the finer topology $\tau^*$.

\begin{thm}[Jankovi\'{c}, Hamlett, \cite{JH}]
	If $X=X^*$  then $f:\langle X, \tau\rangle \to Y$ is $\theta$-continuous iff $f:\langle X, \tau^*\rangle \to Y$ is $\theta$-continuous.
\end{thm}

\begin{defn}[Long \& Herrington, \cite{LongHerr}]
A  function  $f:\langle X, \tau_X\rangle \to \langle Y, \tau_Y \rangle$  is \textbf{faintly continuous}  at  the  point  $x\in  X$   if  and  only  if  for  each  $\theta$-open neighbourhood  $V$  of  $f(x)$  there  is an  open neighbourhood  $U$   of  $x$  such  that  $$f[U]\subseteq V.$$  $f:X\to Y$  is faintly  continuous  if  and  only  if  $f$  is   faintly  continuous  at  each  point  $x  \in   X$.
\end{defn}

Directly from the definition follows that  $f:\langle X, \tau_X\rangle \to \langle Y, \tau_Y \rangle$  is  {faintly continuous} iff $f:\langle X, \tau_X\rangle \to \langle Y, (\tau_\theta)_{Y} \rangle$  is continuous.
In the same paper it is proved that continuity implies $\tau_\theta$-continuity.

\begin{thm}[Long and Herrington, \cite{LongHerr}] \label{T3} If $f:\langle X, \tau_X\rangle \to \langle Y, \tau_Y\rangle$ is continuous then $f:\langle X, (\tau_\theta)_{X}\rangle \to \langle Y, (\tau_\theta)_{Y}\rangle$ is continuous.
\end{thm}

The following result is obvious, but it is given since it will represent one arrow at the diagram which will be given at  the end of the section.

\begin{thm}[Long and Herrington, \cite{LongHerr}] If $f:\langle X, (\tau_\theta)_{X}\rangle \to \langle Y, (\tau_\theta)_{Y}\rangle$ is continuous then   $f:\langle X, \tau_X\rangle \to \langle Y, (\tau_\theta)_{Y}\rangle$ is continuous, i.e. $f:\langle X, \tau_X\rangle \to \langle Y, \tau_Y \rangle$ is  faintly continuous.
\end{thm}

\begin{thm}[Long and Herrington, \cite{LongHerr}] If $f:\langle X, \tau_X\rangle \to \langle Y, \tau_Y\rangle$ is a weakly continuous function then
 $f:\langle X, \tau_X\rangle \to \langle Y, \tau_Y \rangle$ is faintly continuous.
\end{thm}

Trivially, $\theta$-continuous function is weakly continuous. So, so far, the following diagram illustrates current known relations between various types of continuity. It is also known that opposite implications do not hold in general.

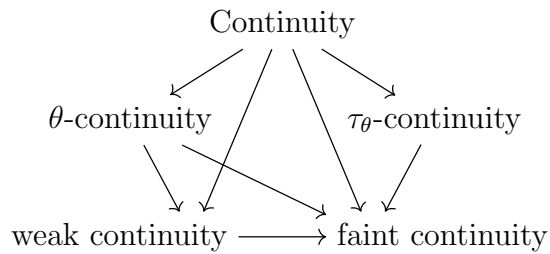
\begin{center}
\begin{tikzpicture}[commutative diagrams/every diagram]
\node(P0)at (90:1.2cm){Continuity};
\node(P1)at (90+72+20:2cm){$\theta$-continuity} ;
\node(P2)at (90+2*72:2cm){\makebox[5ex][r]{weak continuity}};
\node(P3)at (90+3*72:2cm){\makebox[5ex][l]{faint continuity}};
\node(P4)at (90+4*72-20:2cm){$\tau_\theta$-continuity};
\path[commutative diagrams/.cd,every arrow,every label]
(P0)edgenode[swap] { } (P1)
(P1)edgenode[swap] { } (P2)
(P2)edgenode{} (P3)
(P4)edgenode{} (P3)
(P1)edgenode{} (P3)
(P0)edgenode{} (P3)
(P0)edgenode{} (P2)
(P0)edgenode{} (P4);
\end{tikzpicture}
\end{center}

\section{$\theta$-continuity and local closure function}

\begin{thm} \label{Tc1} Let $\langle X, \tau_X, \mathcal{I}_X\rangle$ and $\langle Y, \tau_Y, \mathcal{I}_Y\rangle$ be ideal topological spaces. If $f:\langle X, \tau_X\rangle \to \langle Y, \tau_Y\rangle$ is a $\theta$-continuous function and for all $I\in \mathcal{I}_Y$ we have $f^{-1}[I]\in \mathcal{I}_X$, then there hold the following equivalent conditions:

a) $\forall A \subseteq X~f[\Gamma(A)]\subseteq \Gamma(f[A]);$

b) $\forall B \subseteq Y~\Gamma(f^{-1}[B])\subseteq f^{-1}[\Gamma(B)].$

\end{thm}

\begin{proof}
Let us prove that a) holds. Suppose that there exists $A \subseteq X$  such that there exists $y \in f[\Gamma(A)]\setminus \Gamma(f[A])$. So, there exists $x\in \Gamma(A)$ such that $f(x)=y$ and
\begin{equation}\label{eqthc2}
\forall U \in \tau_X(x)~ \Cl(U) \cap A \not \in \mathcal{I}_X.
 \end{equation}
 Since $y \not \in \Gamma(f[A])$, there exists $W \in \tau_Y(y)$ such that $\Cl(W) \cap f[A] \in \mathcal{I}_Y$. By $\theta$-continuity, there exists $V\in \tau_X(x)$ such that $f[\Cl(V)]\subseteq \Cl[W]$. So $f[\Cl(V)] \cap f[A] \in \mathcal{I}_Y$, implying $f^{-1}[f[\Cl(V)] \cap f[A]] \in \mathcal{I}_X$, and since we have
 $$\Cl(V) \cap A \subseteq f^{-1}[f[\Cl(V)]] \cap f^{-1}[f[A]] \subseteq f^{-1}[f[\Cl(V)] \cap f[A]], $$
we conclude $\Cl(V) \cap A \in \mathcal{I}_X$, which contradicts \eqref{eqthc2}. This proves  a).

 Let us show that b) is equivalent to a). Suppose a) holds and let $B \subseteq Y$.  Then $f[\Gamma(f^{-1}[B])]\subseteq \Gamma(f[f^{-1}[B]])\subseteq \Gamma(B)$. Now  we have $\Gamma(f^{-1}[B]) \subseteq f^{-1}[f[\Gamma(f^{-1}[B])]] \subseteq f^{-1}[\Gamma(B)]$.

 Now suppose b) holds. Then $f^{-1}[\Gamma(f[A])]\supseteq \Gamma(f^{-1}[f[A]]) \supseteq \Gamma(A)$. By taking the image by $f$ of both sets we obtain $\Gamma(f[A]) \supseteq f[f^{-1}[\Gamma(f[A])]]\supseteq f[\Gamma(A)]$.
\end{proof}

In the following theorem we will show how closure in $\sigma$ topology can be obtained  by transfinite recursion.

\begin{thm}\label{Closuresigma}
Let $\CL^0(A)=A$, $\CL^{\alpha+1}(A)=\CL^{\alpha}(A) \cup \Gamma(\CL^{\alpha}(A))$, and $\CL^{\gamma}(A)=\bigcup_{\alpha< \gamma}\CL^\alpha(A)$, for any $A\subset X$, any ordinal $\alpha$ and limit ordinal $\gamma$.  Then

a) For each $\alpha<\beta$, $\CL^\alpha(A) \subseteq \CL^\beta(A)$.

b) For each $\alpha \in ON$, $\CL^\alpha(A) \subseteq \Cl_\sigma(A)$.

c) If there exists $\alpha_0\in ON$ such that $\CL^{\alpha_0}(A)=\CL^{\alpha_0+1}(A)$, then $\CL^{\alpha_0}(A)=\CL^{\alpha}(A)$ for each $\alpha \geq \alpha_0$.

d) There exists $\alpha_0 \in ON$ such that $\CL^{\alpha_0}(A)=\CL^{\alpha}(A)$ for each $\alpha \geq \alpha_0$.

e) For such $\alpha_0$ (and all ordinals larger than it) $\CL^{\alpha_0}(A)=\Cl_\sigma(A)$.
\end{thm}
\begin{proof}
a) Obviously, $\CL^\alpha(A) \subseteq \CL^{\alpha}(A) \cup \Gamma(\CL^{\alpha}(A)) = \CL^{\alpha+1}(A)$ and $\CL^\gamma(A) \supseteq \CL^\alpha(A)$ for limit ordinal $\gamma$ and each $\alpha < \gamma$. So, $\langle \CL^\alpha(A): \alpha \in ON\rangle$ is nondecreasing sequence indexed by the class of all ordinals.

b) Obviously $\Gamma(\Cl_\sigma(A)) \subseteq \Cl_\sigma(A)$ and $A=\CL^0(A) \subseteq \Cl_\sigma(A)$. Applying $\Gamma$ on the  last inclusion we get $\Gamma(\CL^0(A)) \subseteq \Gamma(\Cl_\sigma(A)) \subseteq \Cl_\sigma(A)$. So, $\CL^1(A)\subseteq \Cl_\sigma(A)$. Suppose that for each $\alpha < \beta  $  holds $\CL^\alpha(A) \subseteq \Cl_\sigma(A)$. Let us prove it for $\beta$. If $\beta$ is a limit ordinal, then it holds directly from the property of union, and if $\beta=\delta+1$ for some $\delta\in ON$, then the proof is similar to the case of $\CL^1(A)$.

c) Suppose that, for each $\alpha \in [\alpha_0, \beta)$ we have $\CL^{\alpha_0}(A)=\CL^{\alpha}(A)$, where $\beta > \alpha_0$. Let us prove that it holds for $\beta$.

If $\beta =\delta+1$, then $\CL^{\alpha_0}(A)=\CL^{\delta}(A)$, so $\Gamma(\CL^{\alpha_0}(A))=\Gamma(\CL^{\delta}(A))$, implying $\CL^{\alpha_0+1}(A)=\CL^{\delta+1}(A)$, so $\CL^{\alpha_0}(A)=\CL^{\beta}(A)$.

If $\beta$ is a limit ordinal, then, for each $\alpha\in [\alpha_0, \beta)$ we have $\CL^{\alpha_0}(A)=\CL^{\alpha}(A)$, and, due to the increasing property,
$\CL^{\beta}(A)=\bigcup^{\alpha< \beta}\CL^\alpha(A)=\bigcup \CL^{\alpha_0}(A)=\CL^{\alpha_0}(A)$.

d) Since  $\langle \CL^\alpha(A) : \alpha \in ON\rangle $ is a nondecreasing sequence, it can not strictly increase forever, since there are no more  than $|P(X)|$ different sets. So, there exists $\alpha_0$ such that $\CL^{\alpha_0}(A)=\CL^{\alpha_0+1}(A)$, and d) follows from c).

e) Obviously $A \subseteq \CL^{\alpha_0}(A)\subseteq\Cl_\sigma(A)$. If we prove that $\CL^{\alpha_0}(A)$ is a closed set in topology $\sigma$, the proof is over. Since $\CL^{\alpha_0}(A)= \CL^{\alpha_0+1}(A)=\CL^{\alpha_0}(A) \cup \Gamma(\CL^{\alpha_0}(A))$, we have $\Gamma(\CL^{\alpha_0}(A)) \subseteq \CL^{\alpha_0}(A)$, witnessing that $\Gamma(\CL^{\alpha_0}(A))$ is closed.
\end{proof}

\begin{thm} \label{TTc2} Let $\langle X, \tau_X, \mathcal{I}_X\rangle$ and $\langle Y, \tau_Y, \mathcal{I}_Y\rangle$ be ideal topological spaces. If $f:\langle X, \tau_X\rangle \to \langle Y, \tau_Y\rangle$ is a $\theta$-continuous function and for all $I\in \mathcal{I}_Y$ we have $f^{-1}[I]\in \mathcal{I}_X$, then there hold:

a) $\forall A \subseteq X~f[\CL^{\alpha}(A)] \subseteq \CL^\alpha(f[A])$, for each  ordinal $\alpha$.

b) $\forall A \subseteq X~f[\Cl_{\sigma}(A)] \subseteq \Cl_\sigma(f[A])$;

c) $f:\langle X, \sigma_X\rangle \to \langle Y, \sigma_Y\rangle$ is a continuous function.

\end{thm}
\begin{proof}
a) By definition of $\CL^0$, it holds for $\alpha=0$. Suppose it holds for every $\beta<\alpha$. Let us prove that it holds for $\alpha$. Ih $\alpha$ is a consequtive ordinal, then $\alpha=\delta+1$. So, using Theorem \ref{Tcl},we have \begin{eqnarray*}
f[\CL^{\alpha}(A)]=f[\CL^{\delta+1}(A)]&=&f[\CL^{\delta}(A) \cup \Gamma(\CL^{\delta}(A))]\\&=&f[\CL^{\delta}(A)] \cup f[\Gamma(\CL^{\delta}(A))]\\ &\subseteq& \CL^\delta(f[A]) \cup  \Gamma(f[\CL^{\delta}(A)])\\ &\subseteq& \CL^\delta(f[A]) \cup  \Gamma(\CL^{\delta}(f[A]))\\ &=&\CL^{\delta+1}(f[A])
\\ &=&\CL^{\alpha}(f[A]).
\end{eqnarray*}

If $\alpha$ is a limit ordinal, then
\begin{eqnarray*}
f[\CL^{\alpha}(A)]&=&f[\bigcup_{\gamma <\alpha}\CL^{\gamma}(A)]=\bigcup_{\gamma <\alpha}f[\CL^{\gamma}(A)]
\\&\subseteq&\bigcup_{\gamma <\alpha} \CL^{\gamma}f[A]= \CL^{\alpha}f[A].
\end{eqnarray*}
b) Since, by Theorem \ref{Closuresigma} e), there exists an  ordinal $\alpha_0$ such that  $\CL^{\alpha_0}(A)=\Cl_\sigma(A)$ and ordinal $\alpha_1$ such that $\CL^{\alpha_1}(f[A])=\Cl_\sigma(f[A])$, so, for $\beta =\max \{ \alpha_0,\alpha_1\}$ holds
$$f[\Cl_{\sigma}(A)]=f[\CL^{\beta}(A)] \subseteq \CL^\beta(f[A])= \Cl_\sigma(f[A])$$.

c) is equivalent to  b).
\end{proof}

If we in the previous theorem take $\mathcal{I}=\{\emptyset\}$, we obtain a relation between $\theta$-continuous functions and $\tau_\theta$-continuity.

\begin{cor} If $f:\langle X, \tau_X\rangle \to \langle Y, \tau_Y\rangle$ is a $\theta$-continuous function then
$f:\langle X, (\tau_\theta)_{X}\rangle \to \langle Y, (\tau_\theta)_{Y}\rangle$ is continuous.
\end{cor}

This is an improvement of the result obtained by Long and Herrington  \cite[Th.\ 8]{LongHerr}, stated in Theorem \ref{T3},  which says that continuity implies $\tau_\theta$-continuity.

It is well known that the opposite of the previous corollary does not have to be true.
\begin{ex}\label{ExABC} [$\tau_\theta$-continuity does not imply $\theta$-continuity]
\cite[Ex. 2]{LongHerr} Let $X=\{0,1\}$ with topology $\tau_X=\{\emptyset, \{1\}, \{0,1\}\}$ and let $Y=\{a,b,c\}$ with topology $\tau_Y=\{\emptyset, \{a\}, \{b\}, \{a,b\}, \{a,b,c\}\}$ and $f:X \to Y$ is defined by $f(0)=a$ and $f(1)=b$. Let $x_0=0$. Then $V=\{a\}$ be a neighbourhood of $f(x_0)=a$, and $\Cl(V)=\{a,c\}$. On the other hand, the only neighbourhood of the point $0\in X$ is $U=\{0,1\}$, which is, at the same time, its closure. But $f[\Cl(U)]=f[\{0,1\}]=\{a,b\} \not \subseteq \{a,c\}$, so, $f$ is not $\theta$-continuous. But, the only nonempty $\theta$-open set in $Y$ is $Y$, and its preimage is $X$, which is also $\theta$-open, implying that $f:\langle X, (\tau_\theta)_{X}\rangle \to \langle Y, (\tau_\theta)_{Y}\rangle$ is continuous.
\end{ex}

\section{Weakly continuous functions and local closure function}

\begin{thm} \label{Tw1} Let $\langle X, \tau_X, \mathcal{I}_X\rangle$ and $\langle Y, \tau_Y, \mathcal{I}_Y\rangle$ be ideal topological spaces. If $f:\langle X, \tau_X\rangle \to \langle Y, \tau_Y\rangle$ is a  weakly continuous function and for all $I\in \mathcal{I}_Y$ we have $f^{-1}[I]\in \mathcal{I}_X$, then there hold the following equivalent conditions:

a) $\forall A \subseteq X~f[A^*]\subseteq \Gamma(f[A]);$

 b) $\forall B \subseteq Y~(f^{-1}[B])^*\subseteq f^{-1}[\Gamma(B)].$

\end{thm}

\begin{proof}
Let us prove that a) holds. Suppose that there exists $A \subseteq X$  such that there exists $y \in f[A^*]\setminus \Gamma(f[A])$. So, there exists $x\in A^*$ such that $f(x)=y$. So,
\begin{equation}\label{eqthc1}
\forall U \in \tau_X(x),~ U \cap A \not \in \mathcal{I}_X.
 \end{equation}
 Since $y \not \in \Gamma(f[A])$, there exists $W \in \tau_Y(y)$ such that $\Cl(W) \cap f[A] \in \mathcal{I}_Y$, and by weak continuity, there exists $V\in \tau_X(x)$ such that $f[V]\subseteq \Cl[W]$. So $f[V] \cap f[A] \in \mathcal{I}_Y$, implying $f^{-1}[f[V] \cap f[A]] \in \mathcal{I}_X$, and since we have
 $$V \cap A \subseteq f^{-1}[f[V]] \cap f^{-1}[f[A]] \subseteq f^{-1}[f[V] \cap f[A]], $$
we conclude $V \cap A \in \mathcal{I}_X$, which contradicts \eqref{eqthc1}. This proves  a).

 Let us show that b) is equivalent to a). Suppose a) holds and let $B \subseteq Y$.  Then $f[(f^{-1}[B])^*]\subseteq \Gamma(f[f^{-1}[B]])\subseteq \Gamma(B)$. Now  we have $(f^{-1}[B])^* \subseteq f^{-1}[f[(f^{-1}[B])^*]] \subseteq f^{-1}[\Gamma(B)]$.

 Now suppose b) holds. Then $f^{-1}[\Gamma(f[A])]\supseteq (f^{-1}[f[A]])^* \supseteq A^*$. By taking image by $f$ of both sets we obtain $\Gamma(f[A]) \supseteq f[f^{-1}[\Gamma(f[A])]]\supseteq f[A^*]$.
\end{proof}

\begin{thm} \label{Tw2} Let $\langle X, \tau_X, \mathcal{I}_X\rangle$ and $\langle Y, \tau_Y, \mathcal{I}_Y\rangle$ be ideal topological spaces. If $f:\langle X, \tau_X\rangle \to \langle Y, \tau_Y\rangle$ is a weakly continuous function and for all $I\in \mathcal{I}_Y$ we have $f^{-1}[I]\in \mathcal{I}_X$. Then  $f:\langle X, \tau^*_X\rangle \to \langle Y, \sigma_Y\rangle$ is a continuous function.

\end{thm}
\begin{proof}
Let $A \subset X$. Then its closure in $\tau^*_X$ equals $A \cup A^*$, and by Theorem \ref{Closuresigma} b) we have that closure of $f[A]$ contains $A \cup \Gamma(A)$. By the previous theorem we have that for each $A$ holds $f[A^*]\subseteq \Gamma(f[A])$. Therefore
\begin{eqnarray*}
 f[\Cl_{\tau^*_X}({A})]&=&f[A \cup A^*]=f[A] \cup f[A^*]\\
 &\subseteq&f[A] \cup \Gamma(f[A])\\
 &\subseteq&\Cl_\sigma(f[A]).
 \end{eqnarray*}
\end{proof}

For $\mathcal{I}=\{\emptyset\}$, as a consequence, we obtain an already known result.

\begin{cor} \cite[Th. 10]{LongHerr} If $f:\langle X, \tau_X\rangle \to \langle Y, \tau_Y\rangle$ is a weakly continuous function then
$f:\langle X, \tau_X\rangle \to \langle Y, (\tau_\theta)_{Y}\rangle$ is continuous, which is equivalent to faintly continuity of $f:\langle X, \tau_X\rangle \to \langle Y, \tau_Y \rangle$.
\end{cor}

\begin{ex} [$\tau_\theta$-continuity does not imply weak continuity]
Example \ref{ExABC} also witnesses that continuity of $f:\langle X, (\tau_\theta)_{X}\rangle \to \langle Y, (\tau_\theta)_{Y}\rangle$ does not imply that $f$ is weakly continuous.
\end{ex}
Now, the only open question which needed to be answered to completely fill the diagram given at the end of Section 3 states: Does weak continuity imply $\tau_\theta$-continuity?

We will show that  when either $X$ or $Y$ is finite,  we have the positive answer to the previous question.

\begin{thm}\label{WCtheta}
If $f:\langle X, \tau_X\rangle  \to \langle Y, \tau_Y\rangle $ is weakly continuous and not $\tau_\theta$-continuous, then both $X$ and $Y$ have to be infinite.
\end{thm}
\begin{proof}
Let $f:\langle X, \tau_X \rangle \to \langle Y, \tau_Y \rangle $ be weakly continuous and  not continuous as a function of their $\theta$- topologies. Therefore there exists a set $A\subseteq X$ such that $f[\Cl_{\tau_\theta}(A)] \not \subset \Cl_{\tau_\theta}(f[A])$.    Since $\sigma$ from Theorem \ref{Closuresigma} is equal to $\tau_\theta$ for the trivial ideal $\{\emptyset\}$, and since  $f[A] = f[\CL^0(A)]  \subset \Cl_{\tau_\theta} (f[A])$, there exists $\alpha \in ON$ such that $f[\CL^\alpha(A)]  \subset \Cl_{\tau_\theta}(f[A])$ and $f[\CL^{\alpha+1}(A)] \not \subset \Cl_{\tau_\theta}(f[A])$. So, there exists $x_1\in \CL^{\alpha+1}(A)=\CL^\alpha(A) \cup \Cl_\theta(\CL^\alpha(A))=\Cl_\theta(\CL^\alpha(A))$ such that $y_1=f(x_1) \not \in \Cl_{\tau_\theta}(f[A])$.  For that $y_1\in Y \setminus \Cl_{\tau_\theta}(f[A]) \in (\tau_\theta)_Y$, there exists $V_1 \in \tau_Y(y_1)$ such that $y_1 \in V_1 \subset \Cl(V_1) \subset Y \setminus \Cl_{\tau_\theta}(f[A])$. Due to weak continuity of $f$ there exists $U_1 \in \tau_X(x_1)$ such that
\begin{equation}\label{WCU}
f[U_1] \subseteq \Cl(V_1).
\end{equation}

From $x_1 \in\Cl_\theta(\CL^\alpha(A))$, we conclude that  $\Cl(U_1) \cap  \CL^\alpha(A) \neq \emptyset$. Namely,  let
$x_2 \in \Cl(U_1) \cap \CL^\alpha(A)$. Since $x_2 \in \Cl(U_1)$, we know that
\begin{equation}\label{EqUU1}
\forall U \in \tau_X(x_2)\quad U \cap U_1 \neq \emptyset,
\end{equation}
and if $\tilde{x}\in U \cap U_1$, then $f(\tilde{x}) \not \in \Cl_{\tau_\theta}(f[A])$.

Let $y_2=f(x_2)$. Suppose that  there exists $V \in \tau_Y(y_2) $ such that $V \subseteq  \Cl_{\tau_\theta}({f[A]})$. Then, since $f$ if weakly continuous,  there exists $U_2 \in \tau_X(x_2)$ such that $f[U_2]\subseteq \Cl(V) \subseteq  \Cl_{\tau_\theta}(f[A])$ (since the last one is closed). But this is in contradiction with \eqref{EqUU1} and the remark right after it. So, for each $V \in \tau_Y(y_2) $ there holds
$$V \setminus  \Cl_{\tau_\theta}(f[A]) \in \tau_Y \setminus \{\emptyset\}$$
i.e. $V \setminus  \Cl_{\tau_\theta}(f[A])$ is an nonempty open set disjoint with $ \Cl_{\tau_\theta}(f[A])$.

Let us consider the intersection of all $ V \in \tau_Y(y_2)$, denoted by $O$. Such set  contains $y_2$.
Let us prove that $O$  is not open. If we assume that it is open, we have two possibilities. Firstly $O \subset \Cl_{\tau_\theta}(f[A])$, which we already discussed  is impossible. So, there exists $y \in Y \setminus \Cl_{\tau_\theta}(f[A])$ such that each neighbourhood of $y_2$ intersects $\{y\}$, implying $y_2 \in \Cl(\{y\})\subset \Cl(V_y)$, where $V_y$ is an arbitrary neighbourhood of $y$. This implies that the closure of arbitrary neighbourhood of $y$ intersects $\Cl_{\tau_\theta}(f[A])$, implying $y\in \Cl_{\tau_\theta}(f[A])$, which is impossible. So, as a consequence, we have that $\tau_Y(y_2)$ can not be finite (since the finite intersection of open sets is always open), which implies infinity of $\tau_Y$, and, therefore infinity of $Y$.

Let us suppose that there exists $x_0\in U_1$ such that $x_0 \in U$ for each $ U\in \tau_X(x_2)$. Let $y_0=f(x_0)$. Since $x_0\in U_1$, we have $y_0 \in \Cl(V_1)$. Suppose that there exists $V_0 \in \tau_Y(y_0)$ and $V_2'\in \tau_Y(y_2)$ such that $V_0\cap V_2'=\emptyset$, which implies
\begin{equation}\label{eqExwcnottheta1}
V_0\cap \Cl({V_2'})=\emptyset.
\end{equation}
Then there exists $U_2' \in \tau_X(x_2)$ such that $f[U_2'] \subseteq \Cl({V_2'})$, but this is impossible, since $x_0 \in U_2'$, so, $y_0 \in f[U_2']$, which is in contradiction with \eqref{eqExwcnottheta1}.

Therefore, for each $x \in U_1$ exists $U_x \in \tau_X(x_2)$ such that $x \not \in U$. So, if $U_1$ is finite, then the intersection of all such  $U_x$ is an open set which
does not intersect $U_1$. Therefore, $U_1$ has to be infinite, and there are infinitely many different open sets $U_x$, so $X$ and $\tau_X$ have to be infinite.

\end{proof}

\begin{cor}
If $f:\langle X, \tau_X\rangle  \to \langle Y, \tau_Y\rangle $ is weakly continuous and if $X$ or $Y$ is finite, then $f$ is $\tau_\theta$-continuous.
\end{cor}

The proof of  Theorem \ref{WCtheta} yielded an example of a weakly continuous function which is not $\tau_\theta$-continuous.

\begin{ex}
Let $X=\{x_0,x_1\}\cup \omega$ and $Y=\{y_0,y_1\} \cup \omega\times \{0,1\}$. Let us define $f(x_0)=y_0$, $f(x_1)=y_1$, and  $f(n)=\langle n,1 \rangle$, for $n \in \omega$. Let $\tau_X$ be defined by the neighbourhood base system $$\mathcal{B}_X(x_i)=\{\{x_i\} \cup \omega\setminus K: |K| < \aleph_0\}, \mbox{ for } i\in \{0,1\},\mbox{ and }  \mathcal{B}_X(n)=\{n\}.$$
 and $\tau_Y$  by the neighbourhood base system
%$$\mathcal{B}_Y(y_0)= \{\{y_0\} \cup \{\langle k,0\rangle: k\geq n \in \omega\}: n \in \omega\},$$
%$$\mathcal{B}_Y(y_1)=\{\{y_1\} \cup ((\omega \times  \{1\}) \setminus K): |K| < \aleph_0\},$$
%$$\mathcal{B}_Y(\langle n,0\rangle)=\{\{y_0\} \cup \{\langle k,0\rangle: k \geq n\}\},$$
%$$\mathcal{B}_Y(\langle n,1\rangle)=\{\{y_1, \langle n,0\rangle, \langle n,1\rangle\}\}.$$
\begin{eqnarray*}
\mathcal{B}_Y(y_0)&=& \{\{y_0\} \cup \{\langle k,0\rangle: k\geq n \}: n \in \omega\},\\
\mathcal{B}_Y(y_1)&=&\{\{y_1\} \cup ((\omega \times  \{1\}) \setminus K) \cup  \{\langle n,0\rangle\}: |K| < \aleph_0, n \in \omega\},\\
\mathcal{B}_Y(\langle n,0\rangle)&=&\{\langle n,0\rangle\},\\
\mathcal{B}_Y(\langle n,1\rangle)&=&\{\{y_1\} \cup ((\omega \times  \{1\}) \setminus K) \cup  \{\langle n,0\rangle, \langle n,1\rangle\}: |K| < \aleph_0, n \in \omega\}.
\end{eqnarray*}

Let us prove that $f$ is weakly continuous. We distinguish three cases.

$1^\circ$ $x=x_1$, $f(x_1)=y_1$:  For an arbitrary neighbourhood $V_1=\{y_1\} \cup (\omega \times  \{1\}) \setminus K\cup  \{\langle n,0\rangle\}$, we have $f^{-1}[V_1]=\{x_0\} \cup \omega\setminus K=U_1$, which is open, so $f[U_1] =V_1\subseteq \Cl(V_1) $.

$2^\circ$ $x=n$, $f(n)=\langle n,1\rangle$: This case is trivial, since $\{n\}$ is an open singleton.

$3^\circ$ $x=x_0$, $f(x_0)=y_0$:  For an arbitrary  neighbourhood $V_0=\{y_0\} \cup \{\langle k,0\rangle: k\geq n\}$, let us notice that $\langle k,1\rangle \in \Cl(V_0)$, for each $k \geq n$, since for $V_1=\{y_1\} \cup ((\omega \times  \{1\}) \setminus K) \cup  \{\langle k,0\rangle, \langle k,1\rangle\}$,  a base neighbourhood of $\langle k,1\rangle$, we have $V_0 \cap V_1=\{\langle k,0\rangle\}$, i.e. it is not empty.  So, for $U_0=\{x_0\}\cup \{k:k\geq n\}$ which is an open neighbourhood of $x_0$ in $X$, we have $f[U_0]=\{y_0\} \cup \{\langle k,1\rangle: k\geq n\} \subseteq \Cl(V_0)$.

Finally, let us notice that $\{y_0\}$ is a $\theta$-closed set in $\tau_Y$, since for each other point $y\in Y$, there exist open sets $U_{y_0}$ and $U_y$ such that $y_0\in U_{y_0}$,  $y \in U_y$ and  $U_{y_0}\cap U_y=\emptyset$. On the other hand $f^{-1}[\{y_0\}]=\{x_0\}$ since each neighbourhood of $x_1$ intersects each neighbourhood of $x_0$, implying  $x_0$ is in the closure of each open neighbourhood of $x_1$, so  $x_1\in \Cl_\theta(\{x_0\})$. Therefore, preimage of $\theta$-closed set  $\{y_0\}$ is not closed in  $\theta$-topology, since $\{x_0\} \neq \Cl_\theta(\{x_0\})$.

\end{ex}

So, finally, this example completes our diagram and we conclude that there does not exist other implications between those five types of continuity.

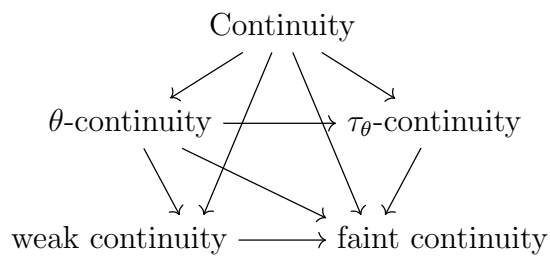
\begin{center}
\begin{tikzpicture}[commutative diagrams/every diagram]
\node(P0)at (90:1.2cm){Continuity};
\node(P1)at (90+72+20:2cm){$\theta$-continuity} ;
\node(P2)at (90+2*72:2cm){\makebox[5ex][r]{weak continuity}};
\node(P3)at (90+3*72:2cm){\makebox[5ex][l]{faint continuity}};
\node(P4)at (90+4*72-20:2cm){$\tau_\theta$-continuity};
\path[commutative diagrams/.cd,every arrow,every label]
(P0)edgenode[swap] { } (P1)
(P1)edgenode[swap] { } (P2)
(P2)edgenode{} (P3)
(P4)edgenode{} (P3)
(P1)edgenode{} (P3)
(P1)edgenode{} (P4)
(P0)edgenode{} (P3)
(P0)edgenode{} (P2)
(P0)edgenode{} (P4);
\end{tikzpicture}
\end{center}

\begin{que}
Is there a nice preserving theorem in ideal topological space, like Theorems \ref{TTc2} and \ref{Tw2}, for faintly continuous functions?
\end{que}

%Also the opposite is not true even for $T_2$-spaces
%\begin{ex}
%We will consider slight modification of the space given in Example \ref{EX_T2}. Let $X=(-\infty,0) \cup \{0_1, 0_2\} \cup (0, \infty)$ and $K=\{\frac 1n: n %\in \mathbb{Z}\setminus \{0\}\}$.  Let the topology $\tau$ be generated by the neighbourhood base:

%If $x<0$, $\mathcal{B}(x)=\{(x-a,x+a)\cap (-\infty,0): a>0\}$.

%If $x>0$, $\mathcal{B}(x)=\{(x-a,x+a)\cap (0,\infty): a>0\}$.

%If $x\in \{0_1,0_2\}$, $\mathcal{B}(x)=((-a,0)\cup (0,a))\setminus K \cup \{0_1,0_2\}$.

%This is $T_1$ space which is not $T_2$?????????????????? (\v Sta sam hteo sa ovim)

%\end{ex}

\end{document}